\documentclass{amsart}
\usepackage{amsfonts}
\usepackage{amsmath,amssymb}
\usepackage{amsthm}
\usepackage{amscd}
\usepackage{graphics}
\usepackage{graphicx}

\theoremstyle{remark}{
\newtheorem{Def}{{\rm Definition}}
\newtheorem{Ex}{{\rm Example}}

\newtheorem{Prob}{{\rm Problem}}

}
\theoremstyle{plain}
{

\newtheorem{Prop}{Proposition}
\newtheorem{Thm}{Theorem}

}

\begin{document}
\title[Morse functions with prescribed preimages on $3$-dimensional manifolds]{On reconstructing Morse-Bott functions with prescribed preimages on $3$-dimensional manifolds and conditions for the reconstruction}
\author{Naoki kitazawa}

\address{Osaka Central Advanced Mathematical Institute (OCAMI) \\
	3-3-138 Sugimoto, Sumiyoshi-ku Osaka 558-8585
	TEL: +81-6-6605-3103
}
\email{naokikitazawa.formath@gmail.com}
\urladdr{https://naokikitazawa.github.io/NaokiKitazawa.html}
\maketitle
\begin{abstract}
We present conditions for reconstruction of Morse-Bott functions with prescribed preimages on $3$-dimensional manifolds. The present work strengthens a previous result for the Morse function case by the author and present a related example as another result. 

This shows a new result on reconstruction of nice smooth functions such that preimages are as prescribed. Such a study has been fundamental, natural, and surprisingly, founded recently, in 2006, by Sharko. Reconstruction of nice smooth functions on closed surfaces has been followed by Masumoto-Saeki, for example. Later, Gelbukh, Marzantowicz, Michalak, and so on, are studying Morse function cases further. The author has started explicit studies for $3$-dimensional cases respecting topologies of preimages of single points and obtained several results. We add another result on this.   
\end{abstract}
\section{Introduction.}
\label{sec:1}

{\it Morse-Bott} functions have been fundamental and important tools and objects in geometry of manifolds.
Our present paper discusses a kind of fundamental problems (Problem \ref{prob:1}). First we explain related fundamental notions and notation.

Let ${\mathbb{R}}^n$ denote the $n$-dimensional Euclidean space. This is a Riemannian manifold with the standard Euclidean metric: let $||p|| \geq 0$ denote the distance between $p \in {\mathbb{R}}^k$ and the origin $0$. Let $\mathbb{R}:={\mathbb{R}}^1$. The $k$-dimensional unit sphere $S^k:=\{x \in {\mathbb{R}}^{k+1}\mid ||x||=1\}$ and the $k$-dimensional unit disk $D^k:=\{x \in {\mathbb{R}}^k \mid ||x|| \leq 1\}$ are defined. For a smooth manifold $X$, let $T_pX$ denote its tangent vector space at $p$. For a smooth map $c:X \rightarrow Y$ between smooth manifolds, a {\it singular point} $p \in X$ of the map means a point where the rank of the differential ${dc}_p:T_pX \rightarrow T_{c(p)}Y$ there, a linear map, is smaller than both the dimensions of the manifolds. The {\it singular value} $c(p)$ of $c$ means a value realized as a value at a singular point of $c$.
A {\it Morse} function $c$ is a smooth function such that their singular points are all in the interior of the manifold and that at each singular point $p$ of $c$, for suitable local coordinates, we have $c(x_1,\cdots x_m)={\Sigma}_{j=1}^{m-i(p)} {x_j}^2-{\Sigma}_{j=1}^{i(p)} {x_{m-i(p)+j}}^2+c(p)$. 
This integer $i(p)$ is uniquely defined and called the {\it index of p for $c$}. 
A {\it Morse-Bott} function is a smooth function which is, at each singular point of $p$, represented as the composition of a smooth map with no singular point (a {\it submersion}) with a Morse function. A Morse function is also a Morse-Bott function: the set of all singular points of it is a smooth regular submanifold with no boundary and in the Morse function case it is a discrete set. \cite{milnor1, milnor2} explain related fundamental theory of Morse functions.
\begin{Prob}
\label{prob:1}
Let $m \geq 2$ be an integer, $a<s<b$ three real numbers, $F_{a}$ and $F_{b}$ {\rm (}$m-1${\rm )}-dimensional smooth closed manifolds.
Can we reconstruct an $m$-dimensional compact and connected manifold ${\tilde{M}}_{F_{a},F_{b}}$ and a Morse(-Bott) function ${\tilde{f}}_{F_{a},F_{b}}:{\tilde{M}}_{F_{a},F_{b}} \rightarrow \mathbb{R}$ with the following properties.
\begin{enumerate}
\item The image is $[a,b]:=\{t \mid a<t<b\}$. There exists a unique singular value of ${\tilde{f}}_{F_{a},F_{b}}:{\tilde{M}}_{F_{a},F_{b}} \rightarrow \mathbb{R}$, which is $s$.
\item The preimage ${{\tilde{f}}_{F_{a},F_{b}}}^{-1}(a)$ (${{\tilde{f}}_{F_{a},F_{b}}}^{-1}(b)$) is diffeomorphic to $F_{a}$ (resp. $F_{b}$) and the preimage ${{\tilde{f}}_{F_{a},F_{b}}}^{-1}(s)$ is connected. 

\item The boundary of the manifold ${\tilde{M}}_{F_{a},F_{b}}$ of the domain is ${{\tilde{f}}_{F_{a},F_{b}}}^{-1}(a) \sqcup {{\tilde{f}}_{F_{a},F_{b}}}^{-1}(b)$.
\end{enumerate}
\end{Prob}

This is fundamental and seems to be classical. However, surprisingly, such a study was founded recently, in 2006, by Sharko (\cite{sharko}). \cite{gelbukh1, gelbukh2, michalak1} have affirmatively solved this in the case $m=2$. See \cite{marzantowiczmichalak, michalak2} as related studies for example. The author has founded this problem for the $m \geq 3$ case and obtained several explicit affirmative results (\cite{kitazawa1, kitazawa2, kitazawa3, kitazawa4}). Among them, we only assume several knowledge and arguments on the article \cite{kitazawa2}. We also refer to the preprint \cite{kitazawa4}.
\subsection{Closed surfaces.}
We assume fundamental knowledge on closed surfaces. We also review several important arguments and facts.
The {\it genus} of a closed, connected and (non-)orientable surface $F$ is a topological invariant for such surfaces.
A closed and connected surface $F$ is diffeomorphic to a surface of the form ${\sharp} {\sharp}_{j_1=1}^{l_1} (S^1 \times S^1) \sharp {\sharp}_{j_2=1}^{l_2} {\mathbb{R}P}^2$.
For the orientable case, the genus is $l_1$ here and $l_2=0$. The genus $0$ case is for the sphere $S^2$. 
For the non-orientable case, the genus is $l_2>0$ under the constraint $l_1=0$. We adopt notation from \cite{kitazawa4} and explain this. We define a non-negative integer $P(F)$ for a closed surface $F$ which may not be connected in the additive way with the following rule for the connected case. We define $P(F)=0$ if $F$ is connected and orientable,and $P(F)>0$ the genus of $F$ if $F$ is connectd and non-orientable.
We can define $P(F):={\Sigma}_{j=1}^{l} P(F_j)$ with $F:={\sqcup}_{j=1} F_j$ where $F_j$ is a connected component of $F$. In this representation let $P_0(F)$ be the number of connected components $F_j$ with $P(F_j)$ being odd.
\subsection{Our result.}
\begin{Thm}[In \cite{kitazawa4} the Morse function case was shown where "$3P(F_a)$" and "$3P(F_b)$" are $P(F_a)$ and $P(F_b)$ respectively: this was also shown to be sufficient in the Morse function case and for this see also \cite{kitazawa2}]
\label{thm:1}
If Problem \ref{prob:1} is solved affirmatively in the case of Morse-Bott functions of the {\it class M}, defined later in Definition \ref{def:1}, with $m=3$, then the value $P_{\rm o}(F_b)-P_{\rm o}(F_a)$ is even and 
we have the relations $P_{\rm o}(F_b) \leq 3P(F_a)$ and $P_{\rm o}(F_a) \leq 3P(F_b)$.
\end{Thm}

We prove this in the next section. Main ingredients of our proof are singular points of Morse-Bott functions and corresponding handles. \cite{milnor2} explains related classical fundamental theory systematically. Note that for the Morse-Bott function case, we can argue similarly and naturally. We expect readers to have related knowledge. 
Most of our arguments here depend on one from \cite{kitazawa4}. However, we also need additional arguments on non-orientable surfaces. We also present an explicit case we have found newly here as Theorem \ref{thm:2}.

\section{A proof of Theorem \ref{thm:1}.}
\subsection{Additional fundamental facts on closed non-orientable surfaces.}
\label{subsec:2.1}
We present additional fundamental arguments on compact non-orientable surfaces.

For a circle $S^1$ smoothly embedded in a closed surface, its closed tubular neighborhood may not be diffeomorphic to $S^1 \times D^1$ and it is a so-called M\"obius band. Let the M\"obius band be denoted by $S^1 \tilde{\times} D^1$. This is also obtained by removing the interior of a smoothly embedded copy $D^2 \subset {\mathbb{R}P}^2$. This is also regarded as a bundle whose fiber is $D^1$ and whose structure group is isomorphic to the group of order $2$. We can define a smooth map $s_{\rm M}:S^1 \rightarrow S^1 \tilde{\times} D^1$ of the bundle such that the composition with the projection of the bundle is the identity map on $S^1$ and that the value at each point is the value corresponding canonically to the origin $0 \in D^1$ of the fiber: this is a so-called section and we call this the {\it canonical section of the M\"obius band}.
\begin{Prop}
\label{prop:1}
\begin{enumerate}
\item \label{prop:1.1} A circle smoothly embedded in ${\mathbb{R}P}^2$ has a closed tubular neighborhood diffeomorphic to $S^1 \times D^1$ if and only if it bounds a smoothly embedded disk $D^2$.
\item \label{prop:1.2}
For a closed surface $F$ with $P(F)=l$ which may not be connected, 
we can choose $l^{\prime}$ mutually disjoint smoothly embedded circles whose closed tubular neighborhoods are diffeomorphic to $S^1 \tilde{\times} D^1$ in $F$ if and only if $l^{\prime}$ is a non-negative integer $0 \leq l^{\prime} \leq l$.
\end{enumerate}
\end{Prop}
\subsection{The manifold ${\tilde{M}}_{F_a,F_b}$ of Problem \ref{prob:1} with $m=3$.}
\label{subsec:2.2}
Singular points of Morse(-Bott) functions and one-to-one 
correspondence between them and handles (resp. families of handles parametrized by smooth regular submanifolds) are important. In short, a singular point of index $k$ for the Morse function on an $m$-dimensional manifold and a so-called {\it$k$-handle}, diffeomorphic to $D^k \times D^{m-k}$ and realized as a smooth submanifold of the manifold of the domain, are corresponded. This is one of important theory in our proof. For the boundary of a manifold $X$, let us use $\partial X$. In terms of (parametrized families of) handles, we discuss the structure of ${\tilde{M}}_{F_a,F_b}$ as we do in \cite{kitazawa4}.
\begin{itemize}
	\item First we prepare the product $F_a \times D^1=F_a \times [0,1]$ where $F_a$ and $F_a \times \{0\}$ are identified by the map $i(x)=(x,0)$.
	\item We choose suitable finitely many disjoint copies smoothly embedded in $F_a \times \{1\}$ in the following.
	\begin{itemize}
		
		\item $D_{1,j}:=D^2 \sqcup D^2$.
		\item $D_{2,j}:=S^1 \times D^1$. 
		\item $D_{3,j}:=S^1 \times (D^1 \sqcup D^1)$.
		\item $D_{4,j}:=S^1 \times D^1$.
	\end{itemize}
	
	\item We attach (parametrized families of) handles as follows. Let the "label $j$" in $D_{i,j}$ ($i=1,2,3,4$) be a positive integer $1 \leq j \leq l_i$ for the order where the integer $l_i$ is the number of copies $D_{i,j}$. 
	\begin{itemize}
		\item  
		We attach a natural family $S^1 \tilde{\times} D^1 \times D^1$ of $1$-handles parametrized by $S_{\rm M}(S^1) \subset S^1 \tilde{\times} D^1$ to $D_{4,j}$ along $\partial(S^1 \tilde{\times} D^1) \times D^1$, which is naturally regarded to be $S^1 \times S^1$, naturally. This corresponds to a circle $C_{j}$ of the set of all singular points of the function ${\tilde{f}}_{F_{a},F_{b}}:{\tilde{M}}_{F_{a},F_{b}} \rightarrow \mathbb{R}$. We attach the families one after another. 
		
		After attaching these families, we have a new $3$-dimensional smooth compact manifold (, smooth the corner,) and let the complementary set of $F_a$ of the boundary be denoted by $F_{s,1}$.  
		\item We attach a natural family $S^1 \times D^1 \times D^1$ of $1$-handles parametrized by $S^1 \times \{0\} \subset S^1 \times D^1$ to $D_{3,j}$ along $S^1 \times \partial D^1 \times D^1=S^1 \times (D^1 \sqcup D^1)$ naturally. This also corresponds to a circle $C_{j,0}$ of the set of all singular points of the function ${\tilde{f}}_{F_{a},F_{b}}:{\tilde{M}}_{F_{a},F_{b}} \rightarrow \mathbb{R}$. We attach the families one after another.
		After attaching these families, we have a new $3$-dimensional smooth compact manifold (, smooth the corner,) and let the complementary set of $F_a$ of the boundary be denoted by $F_{s,2}$.
		\item We attach a $2$-handle $D^2 \times D^1$ to $D_{2,j}$ along $\partial D^2 \times D^1$ naturally. This also corresponds to a singular point of index $2$ for the function ${\tilde{f}}_{F_{a},F_{b}}:{\tilde{M}}_{F_{a},F_{b}} \rightarrow \mathbb{R}$. We attach the handles one after another. The $j$-th handle decomposes a connected component $F_{s,2,j-1,{\rm c}}$ of the existing surface $F_{s,2,j-1}$ into a connected summand of two closed and connected surface or change it into another connected surface: we put $F_{s,2,0}:=F_{s,2}$ and the resulting whole surface obtained from $F_{s,2,j-1}$ is denoted by $F_{s,2,j}$. After attaching the handles, we have a new $3$-dimensional smooth compact manifold (, smooth the corner,) and let the complementary set of $F_a$ of the boundary be denoted by $F_{s,3}$. For example, we have the relations $P(F_{s,3})=P(F_{s,2})-2k_3$ for some non-negative integer $k_3$: this is also a kind of fundamental exercises on topological theory of closed surfaces and presented more precisely in \cite[Our proof of Theorem 1]{kitazawa4}.

		\item We attach a $1$-handle $D^1 \times D^2$ to $D_{1,j}$ along $\partial D^1 \times D^2$ naturally. This also corresponds to a singular point of index $1$ for the function ${\tilde{f}}_{F_{a},F_{b}}:{\tilde{M}}_{F_{a},F_{b}} \rightarrow \mathbb{R}$. We attach the handles one after another. The $j$-th handle connects two components $F_{s,3,j-1,{\rm c}_1}$ and $F_{s,3,j-1,{\rm c}_2}$ of the existing surface $F_{s,3,j-1}$ or change a connected component $F_{s,3,j-1,{\rm c}}$ into another connected surface: we put $F_{s,3,0}:=F_{s,3}$ and the resulting whole surface obtained from $F_{s,3,j-1}$ is denoted by $F_{s,3,j}$. After attaching the handles, we have a new $3$-dimensional smooth compact manifold (, smooth the corner,) and this is diffeomorphic to ${\tilde{M}}_{F_a,F_b}$. The complementary set of $F_a$ of its boundary is $F_b$. For example, we have the relations $P_{\rm o}(F_b)=P_{\rm o}(F_{s,3})-2k_4 \leq P(F_{s,3})-2k_4$ for some non-negative integer $k_3$: this is also a kind of fundamental exercises on topological theory of closed surfaces and presented more precisely in \cite[Our proof of Theorem 1]{kitazawa4}.
	\end{itemize}
	We can regard that the union of all handles here is regarded as the preimage ${{\tilde{f}}_{F_{a},F_{b}}}^{-1}([a+\epsilon,b])$ of the interval $[a+\epsilon,b]$ where $\epsilon>0$ is a sufficiently small number: the preimage ${{\tilde{f}}_{F_{a},F_{b}}}^{-1}(a+\epsilon)$ coincides with $F_a \times \{1\}$.
\end{itemize}
\begin{Def}
\label{def:1}
	In this situation, suppose the following.
\begin{itemize}
\item The copies $D_{3,j}$ are not chosen.
\item Each $D_{4,j}$ is realized as a small closed tubular neighborhood of the boundary of a smoothly embedded copy of $S^1 \tilde{\times} D^1${\rm :} the copy of $S^1 \tilde{\times} D^1$ is denoted by ${D_{4,j}}^{\prime} \subset F_a \times \{1\}$ as a submanifold.
\item The copies ${D_{4,j}}^{\prime} \subset F_a \times \{1\}$ are mutually disjoint.
\end{itemize}
In the situation, the function ${\tilde{M}}_{F_a,F_b}$ is said to be a function of the {\it class M}.
\end{Def}
\subsection{The proof.}
\label{subsec:2.3}
\begin{proof}[Our proof of Theorem \ref{thm:1}]
The value $P_{\rm o}(F_b)-P_{\rm o}(F_a)$ is shown to be even immediately. This comes from a fact on cobordisms: these manifolds $F_a$ and $F_b$ are cobordant and the disjoint union is the boundary $\partial {\tilde{M}}_{F_a,F_b}$.

It is sufficient to check $P_{\rm o}(F_b) \leq 3P(F_a)$: the function $-{\tilde{f}}_{F_{a},F_{b}}:{\tilde{M}}_{F_{a},F_{b}} \rightarrow \mathbb{R}$ explains the case $P_{\rm o}(F_a)<3P(F_b)$ affirmatively.

We abuse notation from our subsection \ref{subsec:2.2}.

The family of handles attached to $D_{4,j}$ decomposes a connected component $F_{a,j-1,{\rm c}}$ of the existing surface $F_{a,j-1}$ into a copy of the component $F_{a,j-1,{\rm c}}$ and a closed, connected and non-orientable surface $K^2$ with $P(K^2)=2$: we put $F_{a,0}:=F_{a} \times \{1\}=F_{a,0}$, the resulting whole surface obtained from $F_{a,j-1}$ is denoted by $F_{a,j}$, and the surface $K^2$ is also a copy of the so-called Klein Bottle. Here we have respected the definition of the function of the class M.
For example, by a fundamental argument, we have the relations $P(F_{s,1})=P(F_a)+2l_4$ and $P_{\rm o}(F_{s,1})=P_{\rm o}(F_a)$. We also have $0 \leq l_4 \leq P(F_a)$ from Proposition \ref{prop:1} (\ref{prop:1.2}).

We do not consider the family of handles attached to $D_{3,j}$ for any $1 \leq j \leq l_3$.

By subsection \ref{subsec:2.2}, we have the inequalities \\
$P_{\rm o}(F_b) \leq  P_{\rm o}(F_{s,3}) \leq P(F_{s,3}) \leq P(F_{s,1})=P(F_{s,2}) \leq P(F_a)+2P(F_a)=3P(F_a)$.
This completes the proof.

\end{proof}
\begin{Ex}
Let $F_a$ be a closed and orientable surface. If for a closed surface $F_b$, $P(F_b)>0$ and even, then we cannot solve Problem \ref{prob:1} affirmatively. As noted in Theorem \ref{thm:1} shortly, we have shown that we cannot solve the Morse function case affirmatively in \cite{kitazawa4}. See also \cite[Example 2]{kitazawa4}, where $F_a:=S^2$ and $F_b:={{\mathbb{R}}P}^2 \sqcup {{\mathbb{R}}P}^2$. 
\end{Ex}
We present Theorem \ref{thm:2}, presenting another simplest example for Theorem \ref{thm:1}.

\begin{Thm}
\label{thm:2}
Let $F_a$ be a closed, connected and non-orientable surface and $F_b:=F_a \sqcup {\sqcup}_{j=1}^{p} {\mathbb{R}P}^2$ where $p$ is an even integer satisfying $1 \leq p \leq 2P(F_a)$. Then Problem \ref{prob:1} is affirmatively solved. Furthermore, for each integer $\frac{p}{2} \leq p^{\prime} \leq P(F_a)$, we can have a desired function ${\tilde{f}}_{F_{a},F_{b}}:{\tilde{M}}_{F_{a},F_{b}} \rightarrow \mathbb{R}$ which is of the class $M$ and the set of all singular points of which is the disjoint union of the following.
\begin{itemize}
\item Exactly $p^{\prime}$ copies of the circle $S^1$.
\item Exactly $p^{\prime}$ singular points of index $2$ for the function.
\end{itemize}
\end{Thm}
\begin{proof}
Here we abuse the notation and rule from our subsection \ref{subsec:2.2} again. 

We choose $p^{\prime}$ disjoint circles in $F_a \times \{1\}$ in Proposition \ref{prop:1} (\ref{prop:1.2}). 
By considering their closed tubular neighborhoods, we have exactly $p^{\prime}$ smooth submanifolds $D_{4,j}$ with the copies ${D_{4,j}}^{\prime}$ of the M\"obius band in Definition \ref{def:1}. 
Remember that each $D_{4,j}$ is chosen as a small collar neighborhood of the boundary $\partial {D_{4,j}}^{\prime}$.
The set $D_{4,j,{\rm o}}:=D_{4,j} \bigcup {D_{4,j}}^{\prime}$ is also a copy of the M\"obius band smoothly embedded in the surface $F_{a,j-1}$ and smoothly isotopic to ${D_{4,j}}^{\prime} \subset F_{a,j-1}$.
We can choose a smooth submanifold $D_{2,j} \subset F_{a,j-1}$ in the interior of each $D_{4,j,{\rm o}}$ in such a way that the following hold.
\begin{itemize}
\item The manifold $D_{2,j}$ is parallel to the small collar neighborhood $D_{4,j}$ of $\partial {D_{4,j}}^{\prime}$ in $F_{a,j-1}$ (where a suitable metric is given in $F_{a,j-1}$: the metric is not essential and here we do not need to understand the notion of parallel subsets in a Riemannian manifold rigorously). 
\item The manifold $D_{2,j}$ is also chosen in $D_{4,j,{\rm o}}-D_{4,j}$.
\end{itemize}
 We do not choose families of handles attached to $D_{1,j}$ or $D_{3,j}$ ($l_1=l_3=0$). We have a desired $3$-dimensional manifold ${\tilde{M}}_{F_a,F_b}$ and its boundary is regarded to be $F_a \sqcup F_b$. We can also have a desired Morse-Bott function ${\tilde{f}}_{F_{a},F_{b}}:{\tilde{M}}_{F_{a},F_{b}} \rightarrow \mathbb{R}$ corresponding naturally to these handles. This completes the proof.

\end{proof}
\section{Conflict of interest and Data availability.}
\noindent {\bf Conflict of interest.} \\
The author has worked at Institute of Mathematics for Industry (https://www.jgmi.kyushu-u.ac.jp/en/about/young-mentors/). This project is closely related to our present study. We thank them for their encouragements. The author is also a researcher at Osaka Central
Advanced Mathematical Institute (OCAMI researcher), supported by MEXT Promotion of Distinctive Joint Research Center Program JPMXP0723833165. He is not employed there. However, we also thank them for such an opportunity. \\
\ \\
{\bf Data availability.} \\
Essentially, data supporting our present study are here. The present paper is also seen as an essential extension of \cite{kitazawa4}. The author hit on an essential idea for the present paper after an earlier version of \cite{kitazawa4} was submitted to a refereed journal. The earlier version of \cite{kitazawa4} has been rejected after the first submission, with several positive comments and no essential errors being found and essentially, the content of present version of \cite{kitazawa4} is essentially same as that of the earlier version.


\begin{thebibliography}{25}
	%

\bibitem{gelbukh1} I. Gelbukh, \textsl{Realization of a digraph as the Reeb graph of a Morse-Bott function on a given surface}, Topology and its Applications, 2024. 
\bibitem{gelbukh2} I. Gelbukh, \textsl{Reeb Graphs of Morse-Bott Functions on a Given Surface}, Bulletin of the Iranian Mathematical Society, Volume 50 Article number 84, 2024. 
\bibitem{golubitskyguillemin} M. Golubitsky and V. Guillemin, \textsl{Stable Mappings and Their Singularities}, Graduate Texts in Mathematics (14), Springer-Verlag (1974).
\bibitem{kitazawa1} N. Kitazawa, \textsl{On Reeb graphs induced from smooth functions on $3$-dimensional closed orientable manifolds with finitely many singular values}, Topol. Methods in Nonlinear Anal. Vol. 59 No. 2B, 897--912.

\bibitem{kitazawa2} N. Kitazawa, \textsl{On Reeb graphs induced from smooth functions on $3$-dimensional closed manifolds which may not be orientable}, Methods of Functional Analysis and Topology Vol. 29 No. 1 (2023), 57--72, 2024.

\bibitem{kitazawa3} N. Kitazawa, \textsl{Realization problems of graphs as Reeb graphs of Morse functions with prescribed preimages}, submitted to another refereed journal, arXiv:2108.06913.
\bibitem{kitazawa4} N. Kitazawa, \textsl{On reconstructing Morse functions with prescribed preimages on $3$-dimensional manifolds and a necessary and sufficient condition for the reconstruction}, a kind of addenda to \cite{kitazawa2}, submitted to another refereed journal based on the rejection of an earlier version with positive comments and no essential errors after the previous submission to a refereed journal, to appear as a prepront arXiv:2412.20626v2, 2025.



\bibitem{marzantowiczmichalak} W. Marzantowicz and L. P. Michalak, \textsl{Relations between Reeb graphs, systems of hypersurfaces and epimorphisms onto free groups}, Fund. Math., 265 (2), 97--140, 2024.

\bibitem{masumotosaeki} Y. Masumoto and O. Saeki, \textsl{A smooth function on a manifold with given Reeb graph}, Kyushu J. Math. 65 (2011), 75--84.
\bibitem{michalak1} L. P. Michalak, \textsl{Realization of a graph as the Reeb graph of a Morse function on a manifold}. Topol. Methods in Nonlinear Anal. 52 (2) (2018), 749--762, arXiv:1805.06727.
\bibitem{michalak2} L. P. Michalak, \textsl{Combinatorial modifications of Reeb graphs and the realization problem}, Discrete Comput. Geom. 65 (2021), 1038--1060, arXiv:1811.08031.
\bibitem{milnor1} J. Milnor, \textsl{Morse Theory}, Annals of Mathematic Studies AM-51, Princeton University Press; 1st Edition (1963.5.1).
\bibitem{milnor2} J. Milnor, \textsl{Lectures on the h-cobordism theorem}, Math. Notes, Princeton Univ. Press, Princeton, N.J. 1965.
\bibitem{sharko} V. Sharko, \textsl{About Kronrod-Reeb graph of a function on a manifold}, Methods of Functional Analysis and
Topology 12 (2006), 389--396.


\end{thebibliography}
\end{document}